\documentclass[12pt]{amsart}
\usepackage{amsfonts}
\usepackage{amscd}
\usepackage{amssymb}
\usepackage{graphics}

\usepackage{amsmath}
\usepackage{graphicx,psfrag}

\usepackage{hyperref}

\usepackage{graphicx,psfrag}
\usepackage{amssymb,amsfonts,amsthm}
\usepackage{amsmath,amsthm,amssymb}
\usepackage{amsfonts}
\usepackage{relsize}
\usepackage[T1]{fontenc}
\usepackage[utf8]{inputenc}
\usepackage{graphicx}
\usepackage{color}
\usepackage{subcaption,mathtools}
\usepackage{epstopdf}

\hypersetup{colorlinks,citecolor=blue}

\newcommand{\la}{\langle}
\newcommand{\ra}{\rangle}
\newtheorem{theorem}{Theorem}

\newtheorem{lemma}[theorem]{Lemma}

\newtheorem{problem}[theorem]{Problem}

\newtheorem{remark}[theorem]{Remark}

\newcommand {\Sph} {\mathrm{Sph}}

\newcommand{\Cl}{\mathrm{Cl}}

\newcommand{\zz}{\mathbb{Z}[\frac{1}{2}]}

\theoremstyle{definition}

\catcode`\@=11
\long\def\@savemarbox#1#2{\global\setbox#1\vtop{\hsize\marginparwidth 
  \@parboxrestore\tiny\raggedright #2}}
\catcode`\@=12

\begin{document}
\author{Gili Golan Polak}
\address{Department of Mathematics, Ben Gurion University of the Negev, Be'er Sheva, Israel}
\email{golangi@bgu.ac.il}
\thanks{The author was partially supported by ISF grant 2322/19.}

\title{Random generation of Thompson group $F$}
\begin{abstract}
We prove that under two natural probabilistic models (studied by Cleary, Elder, Rechnitzer and Taback), the probability of a random pair of elements of Thompson group $F$ generating the entire group is positive. We also prove that for any $k$-generated subgroup $H$ of $F$ which contains a ``natural'' copy of $F$, the probability of a random  $(k+2)$-generated subgroup of $F$ coinciding with $H$ is positive. 
\end{abstract}

\maketitle

\section{Introduction}

The study of random generation of groups has a long history. In 1969, Dixon \cite{D} proved that the probability of $2$ random elements generating the alternating group $\mathrm{Alt}(n)$ tends to $1$ as $n$ goes to infinity. This result was later extended to any sequence of finite simple groups $G_n$ where the order $o(G_n)\rightarrow \infty$. 
Jaikin-Zapirain and Pyber \cite{JP} gave explicit bounds for the number of elements required to generate a finite group $G$ with high probability. As a result, they proved that if $G$ is a finite $d$-generated linear group of dimension
 $n$ then $cd + \log n$ random elements generate $G$ with high probability.

Random generation of finitely generated infinite groups was studied mostly in the profinite case by Kantor and Lubotzky \cite{KL}, Mann \cite{Ma}, Mann and Shalev \cite{MS1} and Jaikin-Zapirain and Pyber \cite{JP}, among others. In the profinite case, the generation is in the topological sense and the probability measure comes from the Haar measure on the group. A profinite group $G$ is \emph{positively finitely generated} or PFG for short, if for any large enough $k$, $k$ random elements of $G$ generate $G$ with positive probability \cite{Ma}. The minimal $k$ such that $k$ random elements of $G$ generate it with positive probability is denoted $d_p(G)$. Kantor and Lubotzky \cite {KL} proved that the free abelian profinite group $\widehat{\mathbb{Z}}^d$ is PFG with $d_p(\widehat{\mathbb{Z}}^d)=d+1$. They also proved that the free profinite group $\widehat{F_d}$ on $d>1$ generators is not PFG.  Mann \cite{Ma} defined positive finite generation of a discrete infinite group by passing to its profinite completion. Using this definition, Mann \cite{Ma} showed that $\mathrm{SL}(n,\mathbb{Z})$ is PFG for any $n\ge 3$ and that any finitely generated virtually solvable group is PFG.

In this paper, we study random generation of Thompson's group $F$. Since the profinite completion of $F$ is $\widehat{\mathbb{Z}}^2$ \cite{CFP}, according to Mann's definition, it is PFG with $d_p(F)=3$. Note however that  
passing to the profinite completion of $F$ 
does not let us distinguish between $F$ and its abelianization $\mathbb{Z}^2$. Moreover, elements of $F$ whose images in the profinite completion generate it, do not necessarily generate $F$ (for example, this is true for any pair of functions $f_1,f_2\in F$ such that $f_1,f_2$ have disjoint supports and such that the slopes $f_1'(0^+)=1, f_1'(1^-)=2$, $f_2'(0^+)=2, f_2'(1^-)=1$).

In this paper, we study random generation of Thompson's group $F$ using the probabilistic models for choosing random subgroups of $F$ introduced and studied in \cite{CERT}. Recall that in \cite{CERT},  Cleary, Elder, Rechnitzer and Taback 
study the likelihood of a random finitely generated subgroup of $F$ being isomorphic to a given subgroup.
The  choice of a random subgroup is done via a choice of its generating set.
To define the ``likelihood'' they use the  definition of asymptotic density, following
 Borovik, Miasnikov and Shpilrain \cite{BMS}:

Let $G$ be an infinite finitely generated group. Let $X$ be a set of representatives of all elements in $G$. That is, $X$ is a set of elements that maps onto $G$. Assume that there is a notion of \emph{size} for elements in $X$. For example, $X$ can be the set of all words over a finite set of generators of $G$ and a natural notion of size in that case is the length of a word in $X$. For all $k\in \mathbb{N}$ we let $X_k$ be the set of all unordered $k$-tuples of elements in $X$. One can associate a notion of size to a $k$-tuple in $X_k$ using the size-notion of elements in $X$. For example, one can define the size of an unordered $k$-tuple $\{x_1,\dots,x_k\}\footnote{In this paper, the notation $\{\cdot\}$ usually stands for a multiset.}$ as the sum of sizes of the elements $x_1,\dots,x_k$. Another option is to consider the maximal size of an element in the tuple. We will consider both of these options below.

Once a notion of size on $X_k$ is fixed, we let $\Sph_k(n)$ be the set of all $k$-tuples in $X_k$ of size $n$. The collection of spheres $\Sph_k(n)$ for all $n\in\mathbb{N}$, whose union covers $X_k$, is called a \emph{stratification} of $X_k$. 
The \emph{asymptotic density} of a subset $T\subseteq X_k$ is defined to be 
$$(*)\  \lim_{n\rightarrow\infty}\frac{|T\cap\Sph_k(n)|}{|\Sph_k(n)|}$$ 
if the limit exists. Regardless of the limit existing, if 
$$\liminf_{n\rightarrow\infty}\frac{|T\cap\Sph_k(n)|}{|\Sph_k(n)|}>0$$ we will say that $T$ has positive asymptotic density, or positive density for short. We would also say that the (asymptotic) probability of a random element of $X_k$ being in $T$ is positive in that case\footnote{Note that by replacing the  limit $(*)$ with the limit over some non-principal ultrfilter $\Omega$ over $\mathbb{N}$, one can always assume that the limit exists. Then one can refer to (finitely additive) probability here.}.

In \cite{CERT}, the authors consider two models for choosing a random $k$-generated subgroup of $F$. 
Recall that each element in $F$ is represented by a unique reduced tree-diagram $(T_+,T_-)$ which consists of two finite binary trees with the same number of carets (see Section \ref{sec:tree}). We let $X$ be the set of reduced tree-diagrams of elements in $F$ and let the size $|g|$ of an element $g=(T_+,T_-)\in X$  be the common number of carets in $T_+$ and $T_-$. The authors use the above two notions of size on $X_k$: in the \emph{sum model} the size of a $k$-tuple in $X_k$ is the sum of sizes of its components and in the \emph{max model} the size of an element in $X_k$ is the maximum size of its	components. 

In \cite{CERT}, a finitely generated subgroup $H$ of $F$ is said to be \emph{persistent} (in a given model) if for every $k$ large enough, the probability of a $k$-generated subgroup of $F$ being isomorphic to $H$ is positive (that is, if the set of $k$-tuples generating a subgroup isomorphic to $H$ has positive density in $X_k$). It is proved in \cite{CERT} that in the sum model, every non-trivial finitely generated subgroup of $F$ is persistent and in the max model, some non-trivial finitely generated subgroups of $F$ are persistent and some are not. 

A group $G$ is said to have a \emph{generic} type of subgroup for some $k\in \mathbb{N}$ if $G$ has a subgroup $H$ such that the asymptotic probability of a random $k$-generated subgroup of $G$ being isomorphic to $H$ is $1$. Note that in both of the above models, for every $k>1$, $F$ does not have a generic type of subgroup \cite{CERT}. In fact, Thompson group $F$ is the first (and so far only) group known to not have a generic  type of subgroup for any $k>1$. 
Jitsukawa \cite{J} proved
that $k$ elements of any finite rank non-abelian free group generically form a free basis for a free
group of rank $k$. Miasnikov and Ushakov \cite{MU} proved this is true also for pure braid groups and right angled Artin groups. Aoun \cite{A} proved the same for non virtually solvable finitely generated linear groups. Gilman, Miasnikov and Osin  proved it for hyperbolic groups and 
Taylor and Tiozo \cite{TT}  proved it for acylindrically hyperbolic groups (see also  Maher and  Sisto \cite{MS}).
 In particular, for all of these groups, a random finitely generated subgroup 
  is almost surely a proper subgroup.
The main result of this paper is the following theorem. 

\begin{theorem}\label{Thm1}
			In both the max-model and the sum-model considered above, the asymptotic probability of a random $k$-generated subgroup of $F$ being equal to $F$ is positive for all $k\ge 2$.
\end{theorem}





By Theorem \ref{Thm1}, a random pair of elements of $F$ generate $F$ with positive asymptotic probability. Other groups where this property holds, with respect to some natural probabilistic model, include $\mathbb{Z}$ (where the probability that two random integers generate $\mathbb{Z}$ is $\frac{6}{\pi^2}$) \cite{AN} and Tarski monsters constructed by Ol'shanskii \cite{O}.  Recall that Tarski monsters are infinite finitely generated non-cyclic groups where every proper subgroup is cyclic\footnote{There are two types of Tarski monsters. One where every proper subgroup is infinite cyclic and one where every proper subgroup is cyclic of order $p$ for some fixed prime $p$.}. In particular, if $T$ is a $2$-generated Tarski monster, then $T$ is generated by any pair of non-commuting elements of $T$. Hence, a random pair of elements of $T$ almost surely generates $T$. Thompson group $F$ is the first example of a finitely presented non virtually-cyclic group where a random pair of elements generate the group with positive asymptotic probability.

We note that the above results for free groups, braid groups, right angled Artin groups and hyperbolic groups were derived in the setting of asymptotic density as described above, where the set $X$ of representatives of elements in the group was taken to be the the set of all finite words over some finite generating set. One can also view the $k$-tuples in this setting as arising from $k$ nearest-neighbor random walks on a Cayley graph of the group.
The results for linear groups and acylindrically hyperbolic groups
 were derived for more general random walks.
  
 Let $S$ be a $2$-generating set of Thompson group $F$. A simple random walk on the Cayley graph of $F$ with respect to $S$ projects onto a simple random walk on  $\mathbb{Z}^2$, the abelianization of $F$.
 It is easy to check that the probability that $2$ independent simple random walks on $\mathbb{Z}^2$ generate  $\mathbb{Z}^2$ is trivial. (Indeed, the main idea is as follows. Given any vector $u\in\mathbb{Z}^2$, 
 if it forms part of a $2$-generating set of $\mathbb{Z}^2$ then there is a vector $v\in\mathbb{Z}^2$ such that the square matrix formed by the vectors $u,v$ has determinant $1$ and such that for any $w\in\mathbb{Z}^2$ the set $\{u,w\}$ generates $\mathbb{Z}^2$ if and only if $w\in\{\pm v+ku:k\in\mathbb{Z}\}$. Using the formula  from \cite[Theorem 1]{G} for enumerating paths on $\mathbb{Z}^2$, one can show that if $S_n$ is a simple random walk on $\mathbb{Z}^2$ then for any pair of vectors $u,v\in\mathbb{Z}^2$, the probability $\mathbb{P}[S_n\in\{\pm v+ku:k\in\mathbb{Z}\}]\le \frac{10}{\sqrt{n}}$ which tends to $0$ as $n\to \infty$).
 Hence, the probability that two independent random walks on the Cayley graph $\mathrm{Cay}(F,S)$ generate Thompson's group $F$ is trivial. We note that two independent simple random walks on $\mathbb{Z}^2$ generate a finite index subgroup of $\mathbb{Z}^2$  with asymptotic probability $1$. Similarly, $3$ independent simple random walks on $\mathbb{Z}^2$ generate it with positive probabilty.  The following problem remains open.

\begin{problem}
	Is it true that any two independent simple random walks on 
	Thompson group $F$  generate a finite index subgroup of $F$ with positive probability? Is there some integer $k>2$ such that $k$ independent random walks on Thompson's group $F$ generate $F$ with positive probability?
\end{problem}


Theorem \ref{Thm1} can be viewed as a stronger version of Theorem 24 of \cite{CERT} which claims that Thompson group $F$ is a persistent subgroup of itself. Let $H\le F$ be a finitely generated subgroup. We say that $H$ is a \emph{perpetual} subgroup of $F$ (in a given model) if for any large enough $k$ the probability of a $k$-generated subgroup of $F$ coinciding with $H$ is positive.
Theorem \ref{Thm1} claims that Thompson group $F$ is a perpetual subgroup of itself. In Section \ref{sec:per} below we generalize this statement: we prove (Theorem \ref{nat}) that any finitely generated subgroup of $F$ which contains a \emph{natural copy} of $F$ (see Section \ref{natural}) is a perpetual subgroup of $F$. 

\vskip .3cm
\textbf{Acknowledgments.} The author would like to thank Mark Sapir for helpful conversations.

\section{Preliminaries on Thompson group $F$}\label{s:FT}

\subsection{F as a group of homeomorphisms}

Recall that Thompson group $F$ is the group of all piecewise linear homeomorphisms of the interval $[0,1]$ with finitely many breakpoints where all breakpoints are finite dyadic and all slopes are integer powers of $2$.  
The group $F$ is generated by two functions $x_0$ and $x_1$ defined as follows \cite{CFP}.
	
	\[
   x_0(t) =
  \begin{cases}
   2t &  \hbox{ if }  0\le t\le \frac{1}{4} \\
   t+\frac14       & \hbox{ if } \frac14\le t\le \frac12 \\
   \frac{t}{2}+\frac12       & \hbox{ if } \frac12\le t\le 1
  \end{cases} 	\qquad	
   x_1(t) =
  \begin{cases}
   t &  \hbox{ if } 0\le t\le \frac12 \\
   2t-\frac12       & \hbox{ if } \frac12\le t\le \frac{5}{8} \\
   t+\frac18       & \hbox{ if } \frac{5}{8}\le t\le \frac34 \\
   \frac{t}{2}+\frac12       & \hbox{ if } \frac34\le t\le 1 	
  \end{cases}
\]

The composition in $F$ is from left to right.

Every element of $F$ is completely determined by how it acts on the set $\zz$. Every number in $(0,1)$ can be described as $.s$ where $s$ is an infinite word in $\{0,1\}$. For each element $g\in F$ there exists a finite collection of pairs of (finite) words $(u_i,v_i)$ in the alphabet $\{0,1\}$ such that every infinite word in $\{0,1\}$ starts with exactly one of the $u_i$'s. The action of $F$ on a number $.s$ is the following: if $s$ starts with $u_i$, we replace $u_i$ by $v_i$. For example, $x_0$ and $x_1$  are the following functions:

\[
   x_0(t) =
  \begin{cases}
   .0\alpha &  \hbox{ if }  t=.00\alpha \\
    .10\alpha       & \hbox{ if } t=.01\alpha\\
   .11\alpha       & \hbox{ if } t=.1\alpha\
  \end{cases} 	\qquad	
   x_1(t) =
  \begin{cases}
   .0\alpha &  \hbox{ if } t=.0\alpha\\
   .10\alpha  &   \hbox{ if } t=.100\alpha\\
   .110\alpha            &  \hbox{ if } t=.101\alpha\\
   .111\alpha                      & \hbox{ if } t=.11\alpha\
  \end{cases}
\]
where $\alpha$ is any infinite binary word.

The group $F$ has the following finite presentation \cite{CFP}.
$$F=\la x_0,x_1\mid [x_0x_1^{-1},x_1^{x_0}]=1,[x_0x_1^{-1},x_1^{x_0^2}]=1\ra,$$ where $a^b$ denotes $b^{-1} ab$. Sometimes, it is more convenient to consider an infinite presentation of $F$. For $i\ge 1$, let $x_{i+1}=x_0^{-i}x_1x_0^i$. In these generators, the group $F$ has the following presentation \cite{CFP}
$$\la x_i, i\ge 0\mid x_i^{x_j}=x_{i+1} \hbox{ for every}\ j<i\ra.$$

\subsection{Elements of F as pairs of binary trees} \label{sec:tree}

Often, it is more convenient to describe elements of $F$ using pairs of 
 (full) finite binary trees $(T_+,T_-)$ which have the same number of leaves. Such a pair is also called a \emph{tree-diagram}.

If $T$ is a finite binary tree, a \emph{branch} in $T$ is a simple path from the root to a leaf. Every non-leaf vertex of $T$ has two outgoing edges: a left edge and a right edge.
If every left edge of $T$ is labeled by $0$ and every right edge is labeled by $1$, then every branch of $T$ is labeled by a finite binary word $u$.  We will usually not distinguish between a branch and its label.

Let $(T_+,T_-)$ be a tree-diagram where $T_+$ and $T_-$ have $n$ leaves. Let $u_1,\dots,u_n$ (resp. $v_1,\dots,v_n$) be the branches of $T_+$ (resp. $T_-$), ordered from left to right.
For each $i=1,\dots,n$ we say that the tree-diagram $(T_+,T_-)$ has the \emph{pair of branches} $u_i\rightarrow v_i$. The function $g$ from $F$ corresponding to this tree-diagram takes binary fraction $.u_i\alpha$ to $.v_i\alpha$ for every $i$ and every infinite binary word $\alpha$. 

The tree-diagrams of the generators of $F$, $x_0$ and $x_1$, appear in Figure \ref{fig:x0x1}.

\begin{figure}[ht]
	\centering
	\begin{subfigure}{.5\textwidth}
		\centering
		\includegraphics[width=.5\linewidth]{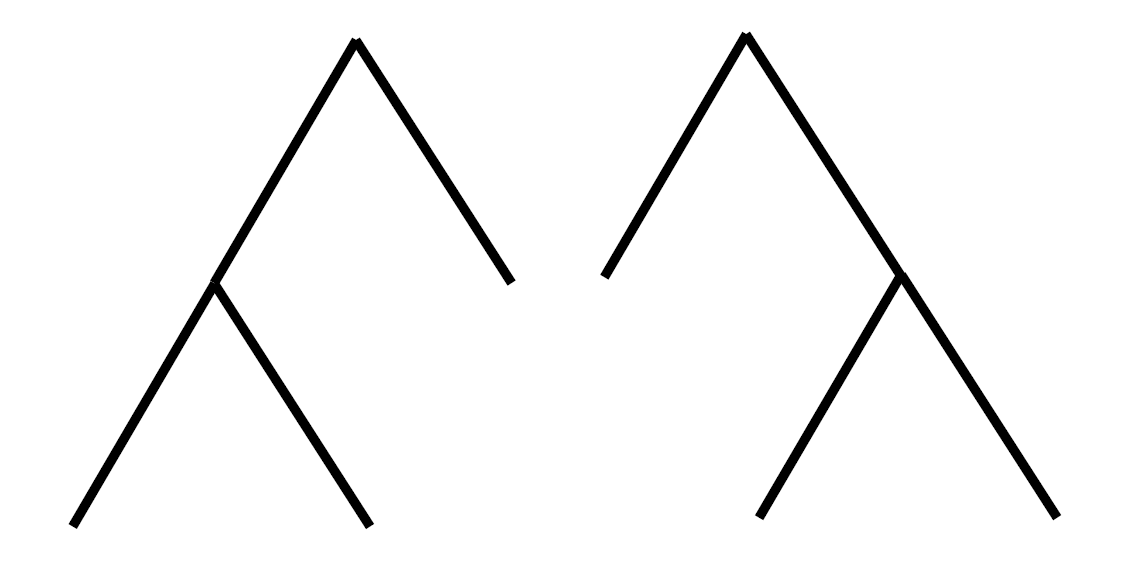}
		\caption{}
		\label{fig:x0}
	\end{subfigure}%
	\begin{subfigure}{.5\textwidth}
		\centering
		\includegraphics[width=.5\linewidth]{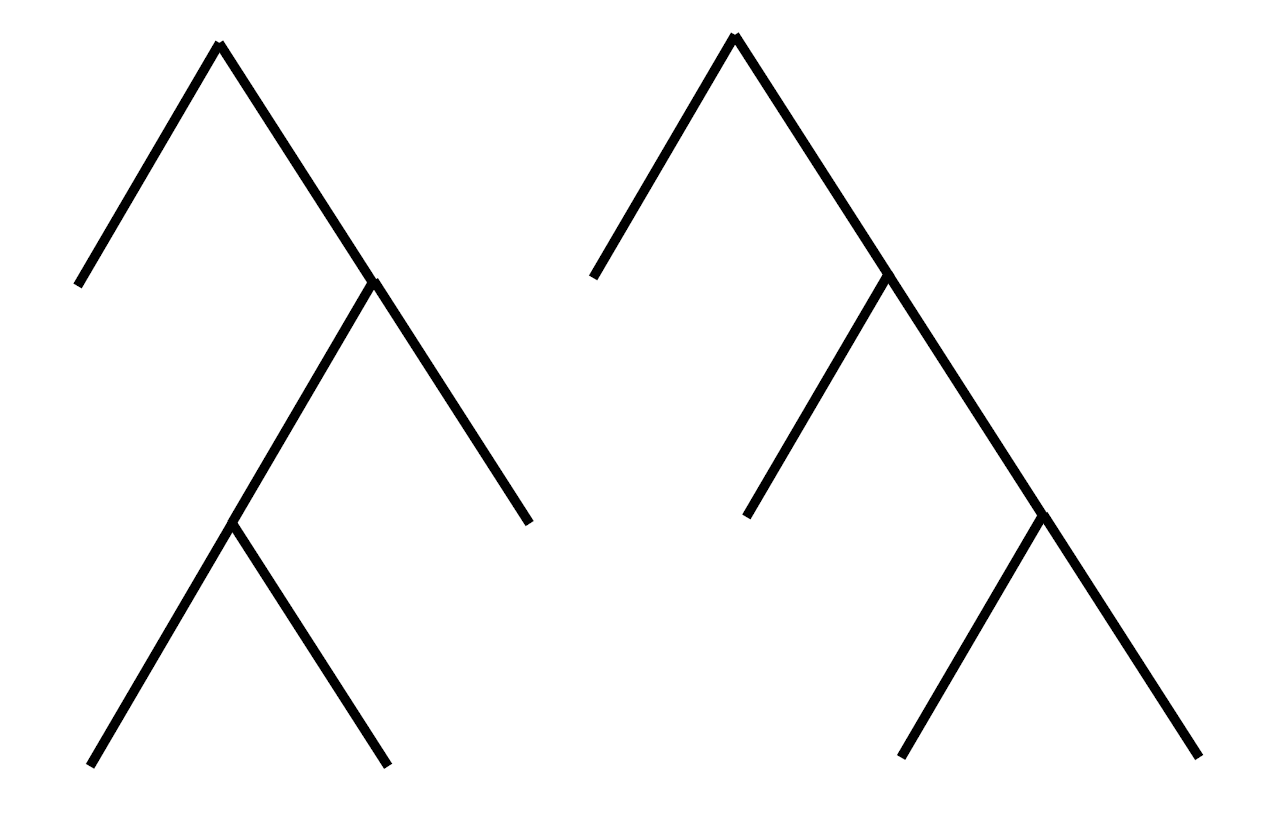}
		\caption{}
		\label{fig:x1}
	\end{subfigure}
	\caption{(A) The tree-diagram of $x_0$. (B) The tree-diagram of $x_1$. In both figures, $T_+$ is on the left and $T_-$ is on the right.}
	\label{fig:x0x1}
\end{figure}

A \emph{caret} is a binary tree composed of a root with two children. If $(T_+,T_-)$ is a tree-diagram and one attaches a caret to the $i^{th}$ leaf of $T_+$ and the $i^{th}$ leaf of $T_-$ then the resulting tree diagram is \emph{equivalent} to $(T_+,T_-)$ and represents the same function in $F$. The opposite operation is that of \emph{reducing} common carets. A tree diagram $(T_+,T_-)$ is called \emph{reduced} if it has no common carets; i.e, if there is no $i$ for which the  $i$ and ${i+1}$ leaves of both $T_+$ and $T_-$ have a common father. 
Every tree-diagram is equivalent to a unique reduced tree-diagram. Thus elements of $F$ can be represented uniquely by reduced tree-diagrams \cite{CFP}. Given an element $g\in F$ we let the \emph{size} of $g$, denoted $|g|$, be the number of carets in the reduced tree-diagram of $g$. 

A slightly different way of describing the function in $F$ corresponding to a given tree-diagram is the following. 
For each finite binary word $u$, we let the \emph{interval associated with $u$}, denoted by $[u]$, be the interval $[.u,.u1^\mathbb N]$. If $(T_+,T_-)$ is a tree-diagram for $f\in F$, we let $u_1,\dots,u_n$ be the branches of $T_+$ and $v_1,\dots,v_n$ be the branches of $T_-$. Then the intervals $[u_1],\dots,[u_n]$ (resp. $[v_1],\dots,[v_n]$) form a subdivision of the interval $[0,1]$. The function $f$ maps each interval $[u_i]$ linearly onto the interval $[v_i]$.

Below, when we say that a function $f$ has a pair of branches $u\rightarrow v$, the meaning is that some tree-diagram representing $f$ has this pair of branches. In other words, this is equivalent to saying that $f$ maps $[u]$ linearly onto $[v]$. In particular, if $f$ has the pair of branches $u\rightarrow v$ then for any finite binary word $w$, $f$ has the pair of branches $uw\rightarrow vw$.

%
The operations in $F$ can be described in terms of operations on tree-diagrams as follows. 

\begin{remark}[See \cite{CFP}]\label{r:000}
	The tree-diagram where both trees are just singletons plays the role of identity in $F$. Given a tree-diagram $(T_+^1,T_-^1)$, the inverse tree-diagram is $(T_-^1,T_+^1)$. If $(T_+^2,T_-^2)$ is another tree-diagram then the product of $(T_+^1,T_-^1)$ and $(T_+^2,T_-^2)$ is defined as follows. There is a minimal finite binary tree $S$ such that $T_-^1$ and $T_+^2$ are rooted subtrees of $S$. 
	Clearly, $(T_+^1,T_-^1)$ is equivalent to a tree-diagram $(T_+,S)$ for some finite binary tree $T_+$. Similarly, $(T_+^2,T_-^2)$ is equivalent to a tree-diagram $(S,T_-)$. The \emph{product}  $(T_+^1,T_-^1)\cdot(T_+^2,T_-^2)$ is (the reduced tree-diagram equivalent to) $(T_+,T_-)$.
\end{remark}

\subsection{The max and sum stratifications}

Let $X$ be the set of all reduced tree-diagrams ($X$ can naturally be identified with $F$). For each $k$, let $X_k$ be the set of all unordered $k$-tuples of elements in $X$.
  
Recall the two stratifications of $X_k$ mentioned above. The \emph{sum stratification} is the stratification of $X_k$ using spheres $\Sph_k^{sum}(n)$ of increasing radii where the size of an unordered $k$-tuple $\{h_1,\dots,h_k\}$ in $X_k$, denoted $||\{h_1,\dots,h_k\}||_{\mathrm{sum}}$, is defined to be the sum of sizes $|h_i|$, for $i=1,\dots,k$. The \emph{max stratification} is defined in a similar way, where the size $||\{h_1,\dots,h_k\}||_{\mathrm{max}}$ of an unordered $k$-tuple in $X_k$ is taken to be the maximum size of any of its components. Let $r_n$ for $n\in\mathbb{N}$ be the number of reduced tree-diagrams in $X$ of size $n$. 
The following is proved in \cite{CERT}.

\begin{lemma}\label{rn}
	The following assertions hold. 
	\begin{enumerate}
		\item[(1)] \cite[Lemma 6]{CERT} For any $k\in\mathbb{Z}$, $$\lim_{n\rightarrow\infty}\frac{r_{n-k}}{r_n}=\mu^{-k},$$ where  $\mu=8+4\sqrt{3}$. 
		\item[(2)] \cite[Lemma 10]{CERT} For $k\ge 1$ and $n\ge k$, the size of the sphere of radius $n$ in $X_k$ with respect to the sum stratification satisfies the following bounds:
		$$r_{n-k+1}\le |\Sph_k^{sum}(n)|\le r_{n+k-1}.$$
		\item[(3)] \cite[Lemma 13]{CERT} For $k\ge 1$ and $n\ge k$, the size of the sphere of radius $n$ in $X_k$ with respect to the max stratification satisfies the following bounds:
		$$\frac{1}{k!}(r_n)^k\le |\Sph_k^{max}(n)|\le k(r_n)^k.$$
	\end{enumerate}
\end{lemma}

\subsection{Natural copies of $F$}\label{natural}

Thompson group $F$ contains many copies of itself (see \cite{Brin}). The copies of $F$ we will be interested in will be of the following simple form. Let $v$ be a finite binary word. We denote by $F_{[v]}$ the subgroup of $F$ of all functions supported on the interval $[v]$; i.e., the subgroup of $F$ of all functions which fix the complement of $[v]$ in $[0,1]$ pointwise.  
Note that $F_{[v]}$ is isomorphic to $F$.  
Indeed, one can define an isomorphism between $F$ and $F_{[v]}$ using tree-diagrams as follows. Let $g$ be an element of $F$ represented by a reduced tree-diagram $(T_+,T_-)$. We map $g$ to an element in $F_{[v]}$, denoted by $g_{[v]}$ and referred to as the \emph{$[v]$-copy of $g$}.
To construct the element $g_{[v]}$ we start with a minimal finite binary tree $T$ which has the branch $v$. Note that the number of carets in $T$ is equal to $|v|$ (i.e., to the length of the word $v$). We take two copies of the tree $T$. To the first copy, we attach the tree $T_+$ at the end of the branch $v$. To the second copy we attach the tree $T_-$ at the end of the branch $v$. The resulting trees are denoted by $R_+$ and $R_-$, respectively. The element $g_{[v]}$ is the one represented by the tree-diagram $(R_+,R_-)$. 
(Note that this tree-diagram is necessarily reduced, so that $|g_{[v]}|=|g|+|v|$.)
 The mapping $g\rightarrow g_{[v]}$ is an isomorphism from $F$ to $F_{[v]}$.
For example, the $[0]$-copies of the generators $x_0,x_1$ of $F$ are depicted in Figure \ref{fig:0x0x1}. It is obvious that these copies generate the subgroup $F_{[0]}$. We call a subgroup $F_{[v]}$ of $F$, a \emph{natural copy} of $F$. 
\begin{figure}[ht]
	\centering
	\begin{subfigure}{.5\textwidth}
		\centering
		\includegraphics[width=.5\linewidth]{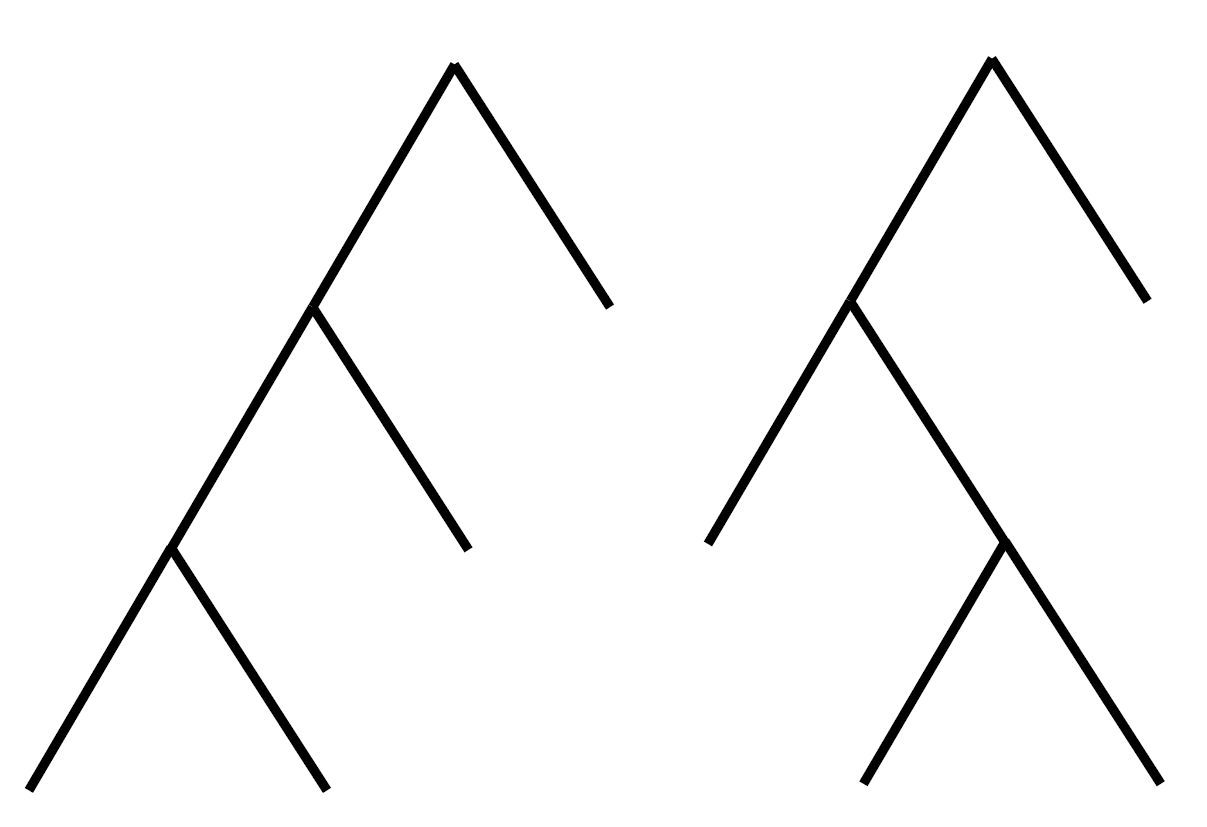}
		\caption{The tree-diagram of $(x_0)_{[0]}$}
		\label{fig:0x0}
	\end{subfigure}%
	\begin{subfigure}{.5\textwidth}
		\centering
		\includegraphics[width=.5\linewidth]{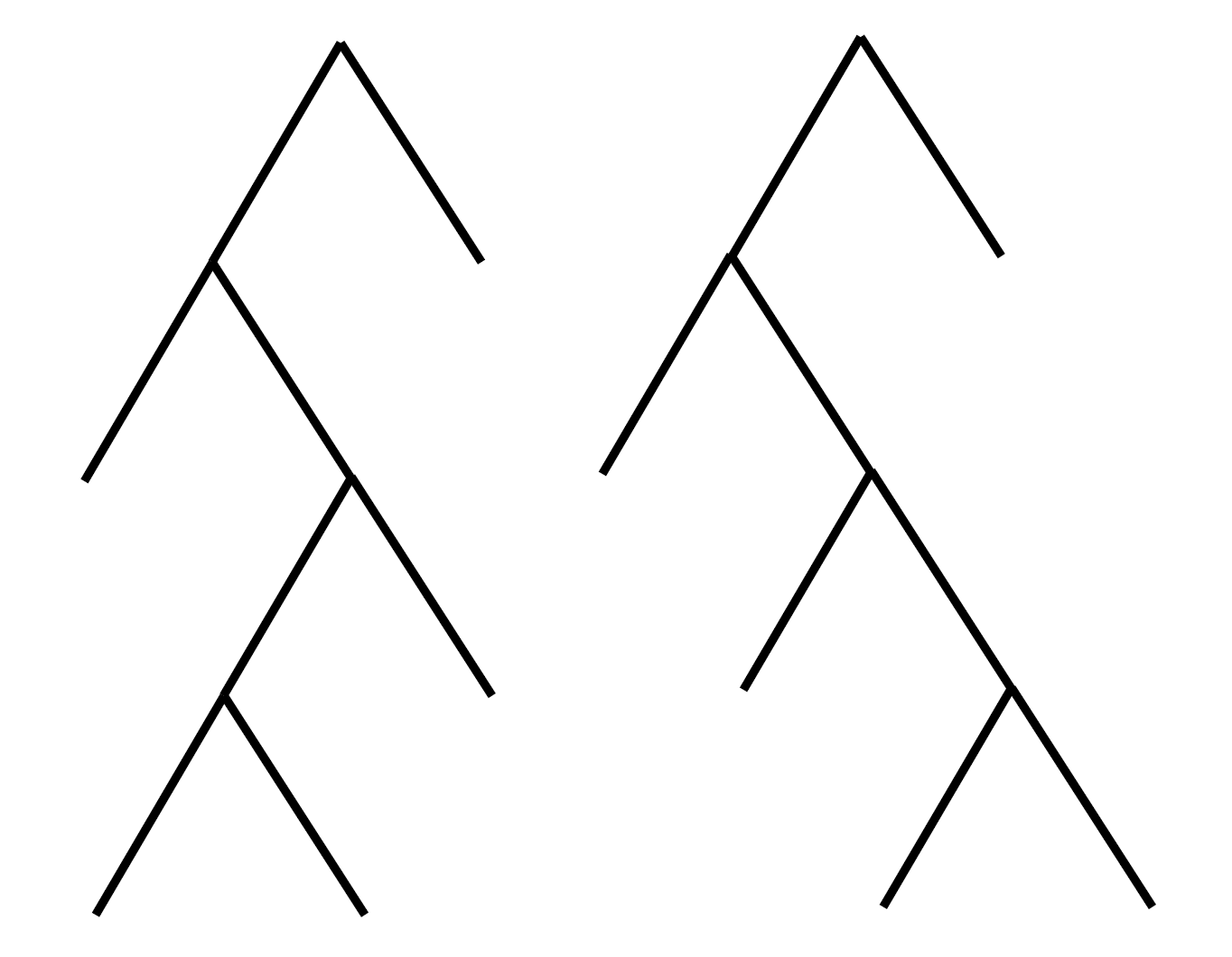}
		\caption{The tree-diagram of $(x_1)_{[0]}$}
		\label{fig:0x1}
	\end{subfigure}
	\caption{}
	\label{fig:0x0x1}
\end{figure}

Below we will multiply elements of $F$ by $[v]$-copies of elements of $F$. We will use the following observation which follows from Remark \ref{r:000} (see also, \cite[Lemma 2.6]{GS18}).

\begin{remark}\label{rem}
	Let $f,g\in F$ be elements with reduced tree-diagrams $(T_+,T_-)$ and $(S_+,S_-)$, respectively. Assume that $(T_+,T_-)$ has the pair of branches $u\rightarrow v$. Then the reduced tree-diagram of the product $f\cdot g_{[v]}$ is the tree-diagram obtained from $(T_+,T_-)$ by attaching the tree $S_+$ to the end of the branch $u$ of $T_+$ and the tree $S_-$ to the end of the branch $v$ of $T_-$. In particular, $|f\cdot g_{[v]}|=|f|+|g|$ and every pair of branches of $(T_+,T_-)$, other than $u\rightarrow v$, is also a pair of branches of $f\cdot g_{[v]}$. 
\end{remark}

%



\subsection{Generating sets of F}

Let $H\le F$. Following \cite{GS,G16}, we define the \emph{closure} of $H$, denoted $\Cl(H)$, to be the subgroup of $F$ of all piecewise-$H$ functions. In \cite{G16}, the author proved that the generation problem in $F$ is decidable. That is, there is an algorithm that decides given a finite subset $X$ of $F$ whether it generates the whole $F$. A recent improvement of \cite[Theorem 7.14]{G16} is the following (as yet unpublished) theorem. 

\begin{theorem}\cite{G20}\label{gen1}
	Let $H$ be a subgroup of $F$. Then $H=F$ if and only if the following conditions are satisfied. 
	\begin{enumerate}
		\item[(1)] $\Cl(H)=F$. 
		\item[(2)] $H[F,F]=F$.
	\end{enumerate}
\end{theorem}

%

Below we  apply Theorem \ref{gen1} to prove that a given subset of $F$ is a generating set of $F$. To verify that Condition (1) in the theorem holds, we will make use of the following observation. The lemma is an immediate result of Remark 7.2 and Lemma 10.6 in \cite{G16}.

\begin{lemma}\label{suffice}
	Let $H$ be a subgroup of $F$. If for each of the following pairs of branches there is an element in $H$ which has the given pair of branches, then $\Cl(H)=F$. 
	\begin{enumerate}
		\item[(1)] $00\rightarrow 0$
		\item[(2)] $11\rightarrow 1$
		\item[(3)] $01\rightarrow 10$
		\item[(4)] $01\rightarrow 010$. 
		\item[(5)] $10\rightarrow 011$. 
	\end{enumerate}
\end{lemma}

\section{Proof of Theorem $1$}

We claim that for every $k\ge 2$ in both the max-model and the sum-model, the probability of a random $k$-generated subgroup of $F$ being equal to $F$ is positive. Let $k\ge 2$. It suffices to prove that the set of all $k$-unordered tuples in $X_k$ which generate $F$ has positive density in $X_k$ (in both models). To prove that, we will consider a subset $S\subseteq X_k$  of unordered tuples of a certain form such that each tuple in $S$ generates $F$. We will prove that with respect to both stratifications the asymptotic density of $S$ in $X_k$ is positive. 

To define $S$ we consider the reduced tree-diagrams of   $x=x_0^2x_1^2x_4^{-1}x_2^{-1}x_1^{-1}x_0^{-2}$ and $y=x_0$.
The reduced tree-diagrams $(T_+(x),T_-(x))$ of $x$ and $(T_+(y),T_-(y))$ of $y$ consist of the following branches.

\[
x:
\begin{cases}
000 & \rightarrow 000\\
00100  & \rightarrow 0010\\
00101 & \rightarrow 00110\\
0011 & \rightarrow 00111\\
01 & \rightarrow 010\\
10 & \rightarrow 011\\
11 & \rightarrow 1\\
\end{cases}\ \ \ 
y :
\begin{cases}
00 & \rightarrow 0\\
01  & \rightarrow 10\\
1 & \rightarrow 11\\
\end{cases} \qquad
\]

We define functions $\phi_i\colon X\to X$ for $i=1,2$ as follows. Let $g\in F$ be an element with reduced tree-diagram $(T_+,T_-)$. We let $\phi_1(g)=x\cdot g_{[00110]}$ and $\phi_2(g)=y\cdot g_{[11]}$ . Note that by Remark \ref{rem}, since $00101 \rightarrow 00110$ is a pair of branches of the reduced tree-diagram of $x$, the image $\phi_1(g)$ is obtained from the tree-diagram $(T_+(x),T_-(x))$ by attaching the tree $T_+$ to the end of the branch $00101$ of $T_+(x)$ and the tree $T_-$ to the end of the branch $00110$ of $T_-(x)$, as depicted in figure \ref{fig:1}. In particular, the mapping $\phi_1$ is injective.  Similarly, since $1\rightarrow 11$ is a pair of branches of the reduced tree-diagram of $y$, the tree-diagram of $\phi_2(g)$ is as depicted in figure \ref{fig:2} and the mapping $\phi_2$ is injective. Note also that $|\phi_1(g)|=|g|+6$ and $|\phi_2(g)|=|g|+2$.

\begin{figure}[ht]
	\centering
	\begin{subfigure}{.6\textwidth}
		\centering
		\includegraphics[width=.6\linewidth]{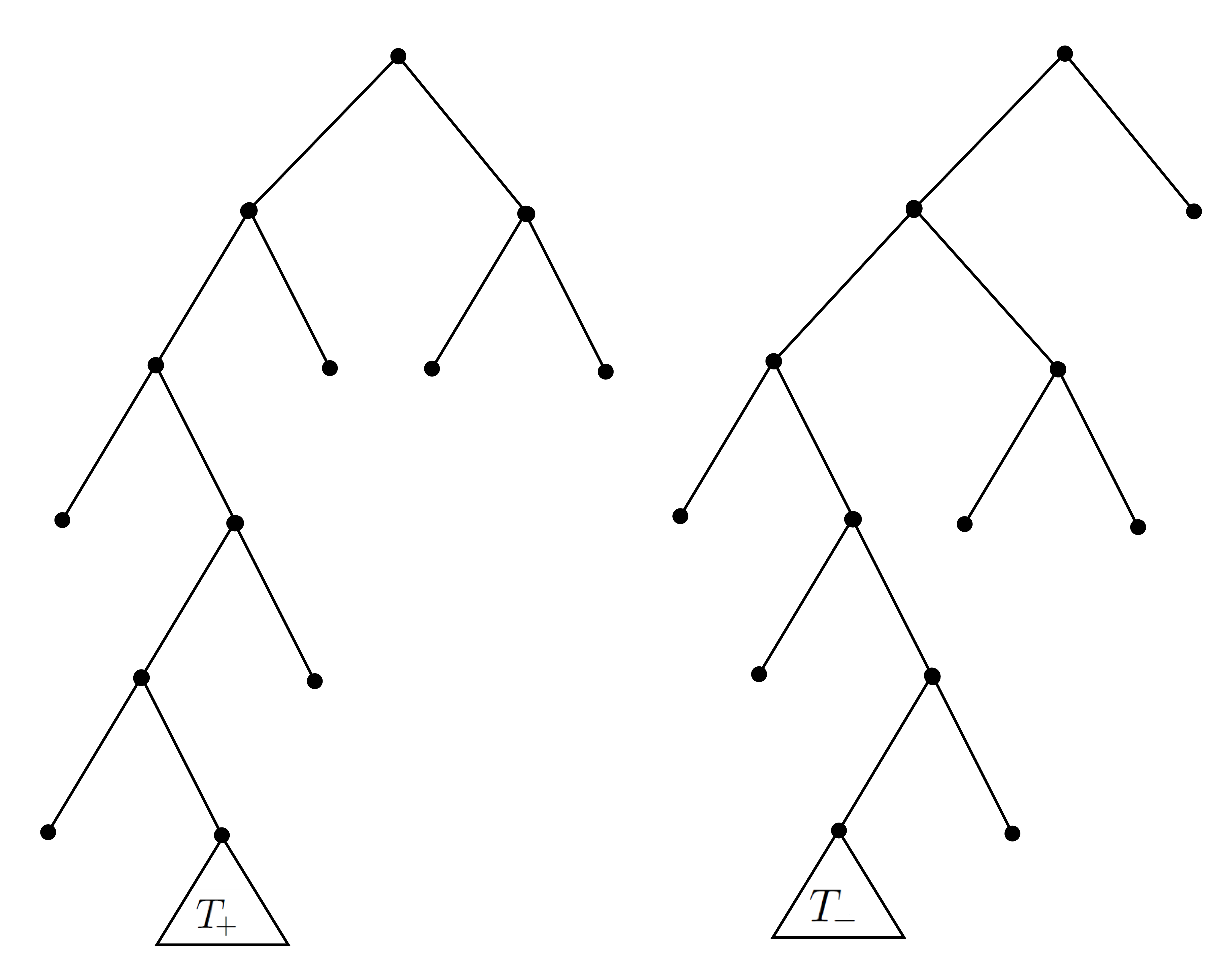}
		\caption{}
		\label{fig:1}
	\end{subfigure}%
	\begin{subfigure}{.5\textwidth}
		\centering
		\includegraphics[width=.5\linewidth]{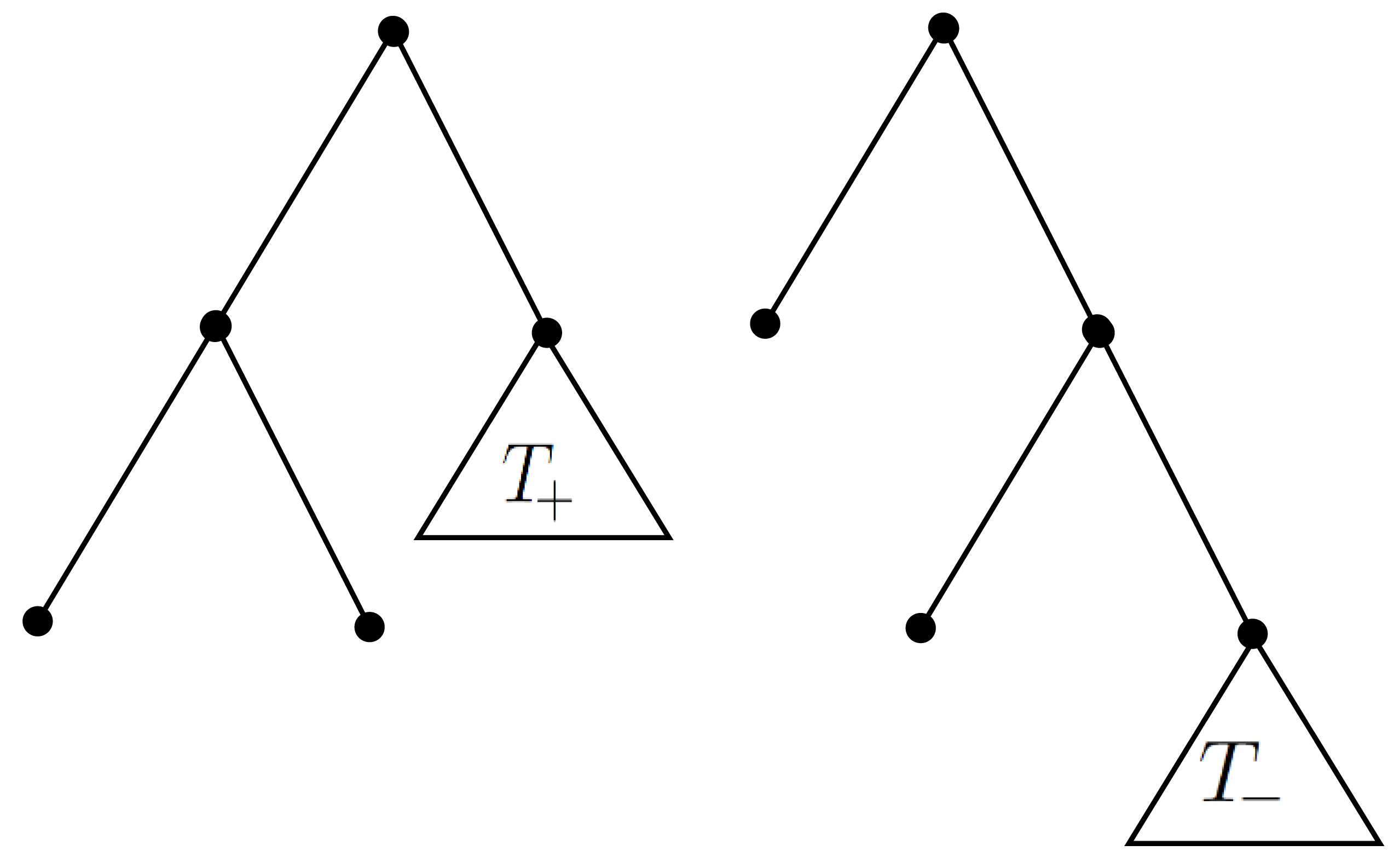}
		\caption{}
		\label{fig:2}
	\end{subfigure}
	\caption{From left to right: the reduced tree-diagrams of $\phi_1(g)$ and $\phi_2(g)$, respectively, for $g$ with reduced tree-diagram $(T_+,T_-)$.}
	\label{fig:12}
\end{figure}




We define a function $\Phi\colon X_k \to X_k$ as follows. We order each $k$-tuple $\{h_1,\dots,h_k\}$ in $X_k$ so that $|h_1|\ge |h_2|\ge\dots\ge|h_k|$. Then we let 
$$\Phi(\{h_1,\dots,h_k\})=\{\phi_1(h_1),\phi_2(h_2),h_3,\dots,h_k\}.$$
Since $|\phi_1(h_1)|>|\phi_2(h_2)|>|h_3|\ge\dots\ge |h_k|$ and the mappings $\phi_1,\phi_2$ are injective, one can easily reconstruct the tuple $\{h_1,\dots,h_k\}$ from its image under $\Phi$. Hence, $\Phi$ is injective. Note also that in the sum model we have 
$$||\Phi(\{h_1,\dots,h_k\})||_{\mathrm{sum}}=||\{h_1,\dots,h_k\}||_{\mathrm{sum}}+8.$$ 
and that in the max model we have 
$$||\Phi(\{h_1,\dots,h_k\})||_{\max}=||\{h_1,\dots,h_k\}||_{\max}+6.$$
We let $S=\Phi(X_k)$.





%

\begin{lemma}\label{lem:SgenF}
	For every $h_1,h_2\in F$, the set $\{\phi_1(h_1),\phi_2(h_2)\}$ generates $F$. In particular, every tuple in $S$ generates $F$. 
\end{lemma}

\begin{proof}
	Let $h_1,h_2\in F$ and let $g_1=\phi_1(h_1)$, $g_2=\phi_2(h_2)$. It suffices to prove that the subgroup $H=\la g_1,g_2\ra$ satisfies $H=F$. To do so, we prove that Conditions (1) and (2) of Theorem \ref{gen1} hold for $H$.

	(1) Note that $g_1$ has the pairs of branches $11\rightarrow 1$,  $10\rightarrow 011$ and  $01\rightarrow 010$. Similarly, $g_2$ has the pairs of branches  $00\rightarrow 0$ and $01\rightarrow 10$ . Hence, by Lemma \ref{suffice}, $\Cl(H)=F$. As such $H$ satisfies Condition (1) of Theorem \ref{gen1}.
	 
	(2) To prove that $H[F,F]=F$ we consider the image of $H$ in the abelianization of $F$. Recall \cite{CFP} that the map $\Gamma\colon F\to \mathbb{Z}^2$ mapping an element $f\in F$ to $(\log_2 f'(0^+),\log_2 f'(1^-))$ is onto and has kernel $[F,F]$. Thus, if $\Gamma$ maps $H$ onto $\mathbb{Z}^2$ then $H[F,F]=F$.
	We note that $\Gamma(g_2)=(1,k)$ for some $k\in\mathbb{Z}$. Indeed, $g_2$ has the pair of branches $00\rightarrow 0$, hence it maps fractions $.00\alpha$ to $.0\alpha$. In particular, it has slope $2^1$ at $0^+$. Similarly, we have $\Gamma(g_1)=\Gamma(x)=(0,1)$. Since $(1,k)$ and $(0,1)$ generate $\mathbb{Z}^2$, we have $H[F,F]=F$, so Condition (2) of Theorem \ref{gen1} holds as well. 
	 
	 
	 Since both the conditions of Theorem \ref{gen1} hold for $H$, we have that $H=F$. 
\end{proof}

In view of Lemma \ref{lem:SgenF}, to finish the proof of Theorem \ref{Thm1} it suffices to prove the following. 

\begin{lemma}\label{pos_den}
	The asymptotic density of $S$ in $X_k$ is positive with respect to both stratifications. As such, the  asymptotic density of the set of $k$-unordered tuples which generate $F$  is positive with respect to both stratifications. 
\end{lemma}

\begin{proof}
	Let us start with the sum stratification.
	As noted above, for each $k$-tuple $\tau$, we have 	$||\Phi(\tau)||_{\mathrm{sum}}=||\tau||_{\mathrm{sum}}+8.$
	Hence, for each $n$ we have  $$\Phi(\Sph_k^{sum}(n))=S\cap \Sph_k^{sum}(n+8).$$
	Since $\Phi$ is injective we have the following. 
	
	$$\liminf_{n\rightarrow\infty}\frac{|S\cap \Sph_k^{sum}(n)|}{|\Sph_k^{sum}(n)|}= \liminf_{n\rightarrow\infty}\frac{|\Sph_k^{sum}(n-8)|}{|\Sph_k^{sum}(n)|}\ge\lim_{n\rightarrow\infty} \frac{r_{n-8-k+1}}{r_{n+k-1}}=\mu^{-2k-6}>0$$
	
	Similarly, for each $k$-tuple $\tau$, we have 	$||\Phi(\tau)||_{\max}=||\tau||_{\max}+6.$
	Hence, for each $n$ we have  $$\Phi(\Sph_k^{max}(n))=S\cap \Sph_k^{max}(n+6).$$
	Since $\Phi$ is injective we have the following. 
	
	$$\liminf_{n\rightarrow\infty}\frac{|S\cap \Sph_k^{max}(n)|}{|\Sph_k^{max}(n)|}= \liminf_{n\rightarrow\infty}\frac{|\Sph_k^{max}(n-6)|}{|\Sph_k^{max}(n)|}\ge\lim_{n\rightarrow\infty} \frac{\frac{1}{k!}(r_{n-6})^k}{k(r_{n})^k}=\frac{1}{k!k}\mu^{-6k}>0$$
Hence, the asymptotic density of $S$ in $X_k$ is positive with respect to both stratifications. 
\end{proof}




 
\section{Perpetual subgroups of Thompson group $F$}\label{sec:per}

Recall that in \cite{CERT}, a finitely generated subgroup $H\le F$ is said to be \emph{persistent}  (in a given model) if for every large enough $k$, the probability of a $k$-generated subgroup of $F$ being isomorphic to $H$ is positive. 
Similarly, we will say that a finitely generated subgroup $H\le F$ is \emph{perpetual} (in a given model) if for any large enough $k$ the probability of a $k$-generated subgroup of $F$ coinciding with $H$ is positive. Clearly, a perpetual subgroup of $F$ (with respect to a given model) is also persistent (with respect to the same model). 
Theorem \ref{Thm1} says that Thompson group $F$ is a perpetual subgroup of itself in both the sum-model and the max-model. 

Recall that Cleary, Elder, Rechnitzer and Taback \cite{CERT} proved that every finitely generated subgroup of $F$ is persistent in the sum-model. In the max-model, they have proved that cyclic subgroups of $F$ are not persistent, so clearly they are not perpetual in that model. More generally, we have the following.

\begin{lemma}\label{abelian}
	Abelian subgroups of $F$ are not perpetual subgroups (in either model). 
\end{lemma}

\begin{proof}
	Let $H$ be a finitely generated abelian subgroup of $F$. We will prove that for every $k\in\mathbb{N}$ the probability of a random $k$-generated subgroup of $F$ being contained in $H$ is zero (in both models). Clearly, we can assume that $H$ is non-trivial. 
	
	Since $H$ is a finitely generated abelian subgroup of  $F$, it is contained in a direct product of finitely many cyclic subgroups of $F$ which have pairwise disjoint supports (see \cite[Theorem 16]{GuSa}). That is, there exists some $m\in\mathbb{N}$ and non-trivial elements $f_1,\dots,f_m\in F$ with pairwise disjoint supports, such that $H$ is contained in $\langle f_1,\dots,f_m\rangle$. In particular, every element of $H$ is the form $f_1^{\ell_1}\cdots f_m^{\ell_m}$ for some $\ell_1,\dots,\ell_m\in\mathbb{Z}$.
	
	It is easy to see that if $f\in F$ is a piecewise linear homeomorphism of $[0,1]$ such that on some interval the slope of $f$ is $2^r$, then the size of $f$ is at least $|r|$ (see, for example, \cite[Lemma 18]{CERT}). It follows that for any
	non-trivial element $f\in F$ and every $n\in\mathbb{Z}$, the size of $|f^n|$ is at least $|n|$. Indeed,  if the slope of the first non-identity linear piece of $f$ is $2^k$ (for $k\neq 0$) then the slope of the first non-identity linear piece of $f^n$ is $2^{kn}$. 
	
	Since the elements $f_1,\dots,f_m$ have disjoint supports, for every $\ell_1,\dots,\ell_m\in\mathbb{Z}$ we have 
		$$(*)\ \ |f_1^{\ell_1}\cdots f_{m}^{\ell_m}|\ge\max\{|\ell_1|,\dots,|\ell_m|\}.$$
	
%
%
%

	
	Now, let $k\in\mathbb{N}$ and let $S$ be the set of all unordered $k$-tuples of elements from $H$. It suffices to prove that the asymptotic density of $S$ in $X_k$ is zero in both models. 
		
	For each $n\in\mathbb{N}$, let $S_n\subseteq S$ be the subset of   all unordered $k$-tuples of elements in $H$ such that all the elements in the tuple are of size at most $n$. We claim that $|S_n|\le (2n+1)^{mk}$. 
	Indeed, if $h$ is an element of $H$ of size at most $n$ then there exist $\ell_1,\dots,\ell_m\in\mathbb{Z}$ such that 
	$h=f_1^{\ell_1}\cdots f_m^{\ell_m}$. By $(*)$, we must have $|\ell_1|,\dots,|\ell_m|\le n$. As such, there are at most $(2n+1)^m$ elements in $H$ of size at most $n$. It follows that the number of unordered $k$-tuples of elements in $H$ of size at most $n$ is bounded from above by $(2n+1)^{mk}$, as claimed. 
	
	Now, let us consider the asymptotic density of $S$ in $X_k$ in the sum model. Note that $S\cap \Sph_k^{sum}(n)\subseteq S_n$. Hence,
	
	$$\lim_{n\rightarrow\infty}\frac{|S\cap \Sph_k^{sum}(n)|}{|\Sph_k^{sum}(n)|}\le \lim_{n\rightarrow\infty}\frac{|S_n|}{|\Sph_k^{sum}(n)|}\le\lim_{n\rightarrow\infty} \frac{(2n+1)^{mk}}{r_{n-k+1}}=0,$$
	where the last equality follows from $r_n$ growing  exponentially (see Lemma \ref{rn}(1)). Hence, the asymptotic density of $S$ in $X_k$ is zero in the sum model. A similar calculation works for the max-model.
\end{proof}

Lemma \ref{abelian} shows that with respect to both models, not all subgroups of $F$ are perpetual. The following theorem gives a wide class of perpetual subgroups of $F$. Recall that a \emph{natural copy} of Thompson's group $F$ is any subgroup of it of the form $F_{[v]}$ for any finite binary word $v$ (see Section \ref{natural}). 

\begin{theorem}\label{nat}
	Let $H$ be a finitely generated subgroup of $F$ which contains a natural copy of $F$. Then $H$ is a perpetual subgroup of $F$ with respect to both models.
\end{theorem}

\begin{proof}
	The proof is similar to the proof of Theorem \ref{Thm1}.  Let $\{f_1,\dots f_m\}$ be a finite generating set of $H$ and assume that $|f_1|\ge|f_2|\ge \cdots\ge |f_m|$. Let $k\ge m+2$. We will prove that the set of unordered $k$-tuples in $X_k$ which generate $H$ has positive density in $X_k$. 
	
	For each $i=1,\dots,m$, let $(T_+^i,T_-^i)$ be the reduced tree-diagram of $f_i$. By assumption, there is a finite binary word $u$ such that $H$ contains the subgroup $F_{[u]}$. 
	We consider the infinite binary word $u0^{\mathbb{N}}$. For each $i=1,\dots,m$, the finite binary tree $T_-^i$ has a unique branch $v_i$ which is a prefix of $u0^{\mathbb{N}}$.
	Let $u_i$ be the branch of $T_+^i$ such that $u_i\rightarrow v_i$  
	is a pair of branches of $(T_+^i,T_-^i)$.
	We let $\ell\ge 0$ be the minimal integer such that for each $i=1,\dots,m$, the finite binary word $v_i$ is a prefix of $u0^{\ell}$ and let $p\equiv\footnote{$\equiv$ denotes letter by letter equality} u0^{\ell}$. Note that $F_{[p]}$ is contained in $F_{[u]}$ and as such, it is a subgroup of $H$. 
	
	For each $i=1,\dots,m$, the word $v_i$ is a prefix of $p$. Let $w_i$ be the finite binary word of length $|p|+7+m-i$ such that $v_iw_i$ is a prefix of $p0^\mathbb{N}$, and let $p_i\equiv v_iw_i$. Note that for each $i=1,\dots,m$, the word $p$ is a prefix of $p_i$.	Note also that for each $i=1,\dots,m$ and for any element $g\in F$, the $[p_i]$-copy of $g$ is the $[v_i]$-copy of the $[w_i]$-copy of $g$. That is $g_{[p_i]}=(g_{[w_i]})_{[v_i]}$. 
	
	Now, let $x$ and $y$ be the elements of $F$ defined in the proof of Theorem \ref{Thm1} and let $\bar{x}=x_{[p]}$ and $\bar{y}=y_{[p]}$ be the $p$-copies of $x$ and $y$. In particular, $|\bar{x}|=|x|+|p|=6+|p|$ and $|\bar{y}|=|y|+|p|=2+|p|$. Note also that $\bar{x}$ has the pair of branches $p00101\rightarrow p00110$  and that $\bar{y}$ has the pair of branches $p1\rightarrow p11$.
		 
	We define several maps from $X$ to itself. 
	\begin{enumerate}
		\item[(1)] For $i=1,\dots,m$, we let $\psi_i\colon X\to X$ be such that for each $g\in F$,  $\psi_i(g)=f_i\cdot g_{[p_i]}$.
		\item[(2)] For $j=1,2$,  we define $\gamma_j\colon X\to X$  as follows: For each $g\in F$ we let 
		$\gamma_1(g)=\bar{x}\cdot g_{[p00110]}$ and
		$\gamma_2(g)=\bar{y}\cdot g_{[p11]}$. 
		\item[(3)] If $k>m+2$ we also define a mapping $\rho\colon X\to X$ which maps every $g\in F$ to its $[p]$-copy, i.e., $\rho(g)=g_{[p]}$.
	\end{enumerate}

	Note that for each $i=1,\dots,m$ and for every $g\in X$ we have
	$\psi_i(g)=f_i \cdot (g_{[w_i]})_{[v_i]}.$
	Let $(T'_+,T'_-)$ be the tree-diagram of the $[w_i]$-copy of $g$. Since $u_i\rightarrow v_i$ is a pair of branches of $f_i=(T_+^i,T_-^i)$, the image $\psi_i(g)$ is obtained from the tree-diagram $(T_+^i,T_-^i)$ by attaching the tree $T'_+$ to the end of the branch $u_i$ of $T_+^i$ and the tree $T'_-$ to the end of the tree $T_-^i$. In particular, (since the mapping $g\mapsto g_{[w_i]}$ is injective) the mapping $\psi_i$ is injective. In addition, for every $g\in X$, by Remark \ref{rem},
	\begin{equation}
	|\psi_i(g)|=|g_{[w_i]}|+|f_i|=|g|+|w_i|+|f_i|=|g|+|f_i|+|p|+7+m-i.
	\end{equation}

	One can verify similarly that the mappings $\gamma_1,\gamma_2$ and $\rho$ are also injective. In addition, for any $g\in X$,
	we have
	\begin{equation}
	\begin{split}
|\gamma_1(g)| & =|g|+|\bar{x}|=|g|+|p|+6
	\end{split}
	\end{equation}
	\begin{equation}
\begin{split}
|\gamma_2(g)|& =|g|+|\bar{y}|=|g|+|p|+2
\end{split}
\end{equation}
	\begin{equation}
\begin{split}
|\rho(g)|& =|g|+|p|
\end{split}
\end{equation}

We define a function $\Gamma\colon X_k \to X_k$ as follows.  We order each $k$-tuple $$\tau=\{h_1,\dots,h_m,h_{m+1},h_{m+2},h_{m+3},\dots,h_k\}\in X_k$$ so that  $|h_1|\ge |h_2|\ge\cdots\ge |h_k|$.
We let 
\begin{equation*}
\Gamma(\tau)= 
\{\psi_1(h_1),\dots,\psi_m(h_m)\}\cup\{\gamma_1(h_{m+1}),\gamma_2(h_{m+2})\}\cup\{\rho(h_{m+3}),\dots,\rho(h_k)\}.
\end{equation*}
Since $|f_1|\ge\cdots\ge|f_m|$ and $|h_1|\ge\cdots\ge|h_k|$,   it follows from Equations (1)-(4) that 
$$|\psi_1(h_1)|>\cdots>|\psi_m(h_m)|>|\gamma_1(h_{m+1})|>|\gamma_2(h_{m+2})|>|\rho(h_{m+3})|\ge\cdots\ge|\rho(h_k)|.$$
Hence, one can easily reconstruct the $k$-tuple $\tau$ from its image under $\Gamma$ (recall that all the mappings $\psi_i,\gamma_j,\rho$ are injective). In particular, the mapping $\Gamma$ is also injective. 

Note also that in the sum model we have 
$$||\Gamma(\tau)||_{\mathrm{sum}}=||\tau||_{\mathrm{sum}}+C_1,$$
where $C_1=k|p|+8+\frac{m}{2}(m+13)+\sum_{i=1}^m|f_i|
$.

Similarly, in the max model, we have 
$$||\Gamma(\tau)||_{\mathrm{max}}=||\tau||_{\mathrm{max}}+C_2,$$
where $C_2=|f_1|+|p|+m+6.$

Now, let $S=\Gamma(X_k)$. Since $C_1$ and $C_2$ are constants determined uniquely by the set of generators $\{f_1,\dots,f_m\}$, the finite binary word $u$ and the integer $k$ we have the following.

\begin{lemma}
	The asymptotic density of $S$ in $X_k$ is positive with respect to both stratifications. 
\end{lemma}

\begin{proof}
	Identical to the proof of Lemma \ref{pos_den}.
\end{proof}

Thus, the following lemma completes the proof of the theorem. 

\begin{lemma}
	Every tuple in $S$ generates $H$.
\end{lemma}

\begin{proof}
	Let $\tau=\{h_1,\dots,h_k\}\in X_k$. It suffices to prove that $$\Gamma(\tau)= 
	\{\psi_1(h_1),\dots,\psi_m(h_m)\}\cup\{\gamma_1(h_{m+1}),\gamma_2(h_{m+2})\}\cup\{\rho(h_{m+3}),\dots,\rho(h_k)\}$$
	generates $H$. Let $K$ be the subgroup of $F$ generated by $\Gamma(\tau)$. We will prove that $K=H$.
	
	First, we claim that $K$ contains $F_{[p]}$. 
	Indeed, let $\phi_1,\phi_2\colon X\to X$ be the functions defined in the proof of Theorem \ref{Thm1}. Note that 
	\begin{equation*}
	\gamma_1(h_{m+1})=\bar{x}\cdot (h_{m+1})_{[p00110]}=x_{[p]}\cdot ((h_{m+1})_{[00110]})_{[p]}=(x\cdot (h_{m+1})_{[00110]} )_{[p]}=(\phi_1(h_{m+1}))_{[p]}
	\end{equation*}
	\begin{equation*}
	\gamma_2(h_{m+2})=\bar{y}\cdot (h_{m+2})_{[p11]}=y_{[p]}\cdot ((h_{m+2})_{[11]})_{[p]}=(y\cdot (h_{m+2})_{[11]} )_{[p]}=(\phi_2(h_{m+2}))_{[p]}
	\end{equation*}
	By Lemma \ref{lem:SgenF}, the elements $\phi_1(h_{m+1}),\phi_2(h_{m+2})$ generate $F$. Hence, their $[p]$-copies, i.e., $\gamma_1(h_{m+1}), \gamma_2(h_{m+2})$ generate $F_{[p]}$. Hence $F_{[p]}$ is contained in $K$. 
	
	Next, we claim that $K\le H$. Recall that $F_{[p]}\le H$. Since $\gamma_1(h_{m+1}),\gamma_2(h_{m+2})\in F_{[p]}$ and the image of $\rho$ is contained in $F_{[p]}$ as well, it 
	suffices to prove that for each $i=1,\dots,m$, $\psi_i(h_i)\in H$. 
	Recall that $\psi_i(h_i)=f_i\cdot g_{[p_i]}$ and that $p$ is a prefix of $p_i$. Hence, $g_{[p_i]}\in F_{[p]}\subseteq H$. Since $f_i\in H$, we have that  $\psi_i(h_i)\in H$ as required. 
	
	To finish, we prove that $H\le K$. It suffices to prove that the generating set $\{f_1,\dots,f_m\}$ is contained in $K$. Let $i\in\{1,\dots,m\}$, and note that $\psi_i(h_i)=f_i\cdot g_{[p_i]}\in K$. Since $g_{[p_i]}\in F_{[p]}$ and we have proved that $F_{[p]}$ is contained in $K$, we get that $f_i\in K$ as required.  	
\end{proof}	
\end{proof}

We finish the paper with the following open problem. 

\begin{problem}
	Which subgroups of $F$ are perpetual with respect to the sum/max-model?  
\end{problem}

It would be especially interesting if there are subgroups which are perpetual with respect to one model but not the other. 

It seems likely that like abelian subgroups of $F$, all finitely generated solvable subgroups of $F$, are perpetual. It is possible that Bleak's characterization of solvable subgroups of $F$ \cite{Bl1,Bl2} (see also \cite{BBH}) can be useful here. 

In fact, we believe that the finitely generated subgroups of $F$ which are perpetual with respect to both models are exactly those which contain a natural copy of $F$. A possible candidate for a counter example is the subgroup $B$ of $F$ constructed in the proof of \cite[Theorem 9.1]{G16}. The group $B$ is an elementary amenable subgroup of $F$ (it is a copy of the Brin-Navas group \cite[Section 5]{Brin1},\cite[Example 6.3]{N}) and it is maximal inside a normal subgroup $K\triangleleft F$, such that $F$ is a cyclic extension of $K$.

\end{document}